\documentclass[twoside,12pt]{article}

\usepackage{times,amsmath,amssymb,amsthm,cite}

\newcommand{\subj}[1]{\par\noindent{\bf AMS Subject Classifications: }#1.}
\newcommand{\keyw}[1]{\par\noindent{\bf Keywords: }#1.}

\numberwithin{equation}{section}
\numberwithin{figure}{section}

\newtheorem{theorem}{Theorem}[section]
\newtheorem{lemma}[theorem]{Lemma}
\newtheorem{corollary}[theorem]{Corollary}

\theoremstyle{definition}
\newtheorem{definition}[theorem]{Definition}
\newtheorem{example}[theorem]{Example}
\theoremstyle{remark}
\newtheorem{remark}[theorem]{Remark}

\everymath{\displaystyle}
\date{}
\newcommand{\ijde}
{\vspace{-1in}\normalsize\flushleft
This is a preprint of a paper whose final and definite form will be published in:\\
Int. J. Difference Equ. ({\tt http://campus.mst.edu/ijde}).\\
Submitted Aug 27, 2012; Revised Nov 14, 2012; Accepted Nov 15, 2012.\\\vspace{1mm}\hrule\vspace{5mm}}

\begin{document}

\title{\ijde \center\Large\bf The Delta-nabla Calculus of Variations
for Composition Functionals on Time Scales}

\author{{\bf Monika Dryl} and {\bf Delfim F. M. Torres}\\
{\tt \{monikadryl, delfim\}@ua.pt}\\[0.3cm]
Center for Research and Development in Mathematics and Applications\\
Department of Mathematics, University of Aveiro\\
3810-193 Aveiro, Portugal}

\maketitle

\thispagestyle{empty}

% ------------------------------------------------

\begin{abstract}
We develop the calculus of variations on time scales
for a functional that is the composition of a certain scalar function
with the delta and nabla integrals of a vector valued field.
Euler--Lagrange equations, transversality conditions,
and necessary optimality conditions for isoperimetric problems,
on an arbitrary time scale, are proved.
Interesting corollaries and examples are presented.
\end{abstract}

% ------------------------------------------------

\subj{26E70, 49K05}

\keyw{calculus of variations, optimality conditions, time scales}

% ------------------------------------------------

\section{Introduction}

We study a general problem of the calculus of variations
on an arbitrary time scale $\mathbb{T}$. More precisely,
we consider the problem of extremizing (i.e., minimizing or maximizing)
a delta-nabla integral functional
\begin{multline*}
\mathcal L(x)=H\left(\int\limits_{a}^{b}f_{1}(t,x^{\sigma}(t),x^{\Delta}(t))\Delta t,
\ldots,\int\limits_{a}^{b}f_{k}(t,x^{\sigma}(t),x^{\Delta}(t))\Delta t,\right.\\
\left.\int\limits_{a}^{b}f_{k+1}(t,x^{\rho}(t),x^{\nabla}(t))\nabla t,\ldots,
\int\limits_{a}^{b}f_{k+n}(t,x^{\rho}(t),x^{\nabla}(t)) \nabla t\right)
\end{multline*}
possibly subject to boundary conditions and/or isoperimetric constraints.
For the interest in studying such type of variational problems in economics,
we refer the reader to \cite{MalinowskaTorresCompositionDelta} and references therein.
For a review on general approaches to the calculus of variations on time scales,
which allow to obtain both delta and nabla variational calculus
as particular cases, see \cite{GirejkoMalinowska,TorresDeltaNabla,MyID:212}.
Throughout the text we assume the reader to be familiar
with the basic definitions and results of time scales
\cite{BohnerDEOTS,MBbook2003,Hilger90,Hilger97}.

The article is organized as follows.
In Section~\ref{sec:prelim} we collect some necessary definitions
and theorems of the nabla and delta calculus on time scales.
The main results are presented in Section~\ref{sec:mr}.
We begin by proving general Euler--Lagrange equations (Theorem~\ref{main}).
Next we consider the situations when initial or terminal boundary conditions are not specified,
obtaining corresponding transversality conditions (Theorems~\ref{punkt pocz} and \ref{punkt kon}).
The results are applied to quotient variational problems in Corollary~\ref{cor iloraz}.
Finally, we prove necessary optimality conditions for general isoperimetric problems
given by the composition of delta-nabla integrals (Theorem~\ref{twierdzenie iso}).
We end with Section~\ref{sec:examples}, illustrating the new results of the paper
with several examples.

% -------------------------------------

\section{Preliminaries}
\label{sec:prelim}

In this section we review the main results necessary in the sequel.
For basic definitions, notations and results of the theory of time scales,
we refer the reader to the books \cite{BohnerDEOTS,MBbook2003}.

The following two lemmas are
the extension of the Dubois--Reymond
fundamental lemma of the calculus of variations \cite{book:vanBrunt}
to the nabla (Lemma~\ref{Dubois-Reymond nabla})
and delta (Lemma~\ref{Dubois-Reymond delta})
time scale calculus. We remark that all intervals
in this paper are time scale intervals.

\begin{lemma}[\cite{natalia:CV}]
\label{Dubois-Reymond nabla}
Let $f\in C_{ld}([a,b],\mathbb{R})$. If
$$
\int\limits_{a}^{b}f(t)\eta^{\nabla}(t)\nabla t=0 \
\textrm{ for all } \ \eta \in C^{1}_{ld}([a,b],\mathbb{R}) \
\textrm{ with } \
\eta(a)=\eta(b)=0,
$$
then $f(t)=c$, for some constant $c$, for all $t\in [a,b]_{\kappa}$.
\end{lemma}

\begin{lemma}[\cite{Bohner:CV}]
\label{Dubois-Reymond delta}
Let $f\in C_{rd}([a,b],\mathbb{R})$. If
$$
\int\limits_{a}^{b}f(t)\eta^{\Delta}(t)\Delta t=0 \
\textrm{ for all } \ \eta \in C^{1}_{rd}([a,b],\mathbb{R}) \
\textrm{ with }\
\eta(a)=\eta(b)=0,
$$
then $f(t)=c$, for some constant $c$, for all $t\in [a,b]^{\kappa}$.
\end{lemma}

Under some assumptions, it is possible to relate the delta and nabla derivatives
(Theorem~\ref{polaczenie rozniczka delta i nabla}) as well as
the delta and nabla integrals (Theorem~\ref{polaczenie calka delta i nabla}).

\begin{theorem}[\cite{AticiGreen:functions}]
\label{polaczenie rozniczka delta i nabla}
If $f:\mathbb{T}\rightarrow\mathbb{R}$ is delta differentiable
on $\mathbb{T}^{\kappa}$ and $f^{\Delta}$ is continuous on $\mathbb{T}^{\kappa}$,
then $f$ is nabla differentiable on  $\mathbb{T}_{\kappa}$ and
\begin{equation}
\label{polaczenie rozniczka nabla}
f^{\nabla}(t)=(f^{\Delta})^{\rho}(t) \
\textrm{ for all } \
t\in\mathbb{T}_{\kappa}.
\end{equation}
If $f:\mathbb{T}\rightarrow\mathbb{R}$ is nabla differentiable on
$\mathbb{T}_{\kappa}$ and $f^{\nabla}$ is continuous on $\mathbb{T}_{\kappa}$,
then $f$ is delta differentiable on $\mathbb{T}^{\kappa}$ and
\begin{equation}
\label{polaczenie rozniczka delta}
f^{\Delta}(t)=(f^{\nabla})^{\sigma}(t) \
\textrm{ for all } \
t\in\mathbb{T}^{\kappa}.
\end{equation}
\end{theorem}

\begin{theorem}[\cite{G:G:S}]
\label{polaczenie calka delta i nabla}
Let $a,b\in\mathbb{T}$ with $a<b$.
If function $f:\mathbb{T}\rightarrow\mathbb{R}$ is continuous, then
\begin{equation}
\label{polaczenie calka delta}
\int\limits_{a}^{b}f(t)\Delta t=\int\limits_{a}^{b}f^{\rho}(t)\nabla t,
\end{equation}
\begin{equation}
\label{polaczenie calka nabla}
\int\limits_{a}^{b}f(t)\nabla t=\int\limits_{a}^{b}f^{\sigma}(t)\Delta t.
\end{equation}
\end{theorem}

% -------------------------------------

\section{Main results}
\label{sec:mr}

By $C_{k,n}^{1}([a,b],\mathbb{R})$ we denote the class of functions
$x:[a,b]\rightarrow\mathbb{R}$ such that:
if $n=0$, then $x^{\Delta}$ is continuous on $[a,b]^{\kappa}$;
if $k=0$, then $x^{\nabla}$ is continuous on $[a,b]_{\kappa}$;
if $k\neq 0$ and $n\neq 0$, then
$x^{\Delta}$ is continuous on $[a,b]^{\kappa}_{\kappa}$ and
$x^{\nabla}$ is continuous on $[a,b]_{\kappa}^{\kappa}$,
where $[a,b]^{\kappa}_{\kappa}:=[a,b]^{\kappa}\cap [a,b]_{\kappa}$.
We consider the following problem of calculus of variations:
\begin{multline}
\label{problem}
\mathcal L(x)=H\left(\int\limits_{a}^{b}f_{1}(t,x^{\sigma}(t),x^{\Delta}(t))\Delta t,
\ldots,\int\limits_{a}^{b}f_{k}(t,x^{\sigma}(t),x^{\Delta}(t))\Delta t,\right.\\
\left.\int\limits_{a}^{b}f_{k+1}(t,x^{\rho}(t),x^{\nabla}(t))\nabla t,\ldots,
\int\limits_{a}^{b}f_{k+n}(t,x^{\rho}(t),x^{\nabla}(t))\nabla t\right)
\longrightarrow \textrm{extr},
\end{multline}
\begin{equation}
\label{punkty}
(x(a)=x_{a}), \quad (x(b)=x_{b}),
\end{equation}
where ``extr'' means ``minimize'' or ``maximize''.
The parentheses in \eqref{punkty}, around the end-point conditions,
means that those conditions may or may not occur
(it is possible that both $x(a)$ and $x(b)$ are free).
A function $x\in C_{k,n}^{1}$ is said to be admissible provided
it satisfies the boundary conditions \eqref{punkty} (if any is given).
For $k = 0$ problem \eqref{problem} reduces to a nabla problem
(no delta integral and delta derivative is present);
for $n = 0$ problem \eqref{problem} reduces to a delta problem
(no nabla integral and nabla derivative is present). We assume that:
\begin{enumerate}
\item the function $H:\mathbb{R}^{n+k}\rightarrow\mathbb{R}$
has continuous partial derivatives with respect to its arguments,
which we denote by $H_{i}^{'}$, $i=1,\ldots,n+k$;
\item functions $(t,y,v)\rightarrow f_{i}(t,y,v)$
from $[a,b]\times\mathbb{R}^{2}$ to $\mathbb{R}$, $i=1,\ldots, n+k$,
have partial continuous derivatives with respect to $y$ and $v$
for all $t \in [a,b]$, which we denote by $f_{iy}$ and $f_{iv}$;
\item $f_{i}$, $f_{iy}$, $f_{iv}$ are continuous
on $[a,b]^{\kappa}$, $i=1,\ldots,k$, and continuous on
$[a,b]_{\kappa}$, $i=k+1,\ldots,k+n$, for all $x\in C_{k,n}^{1}$.
\end{enumerate}

The following norm in $C_{k,n}^{1}$ is considered:
$$
||x||_{1,\infty}:=||x^{\sigma}||_{\infty}+||x^{\Delta}||_{\infty}
+||x^{\rho}||_{\infty}+||x^{\nabla}||_{\infty},
$$
where $||x||_{\infty}:= \sup |x(t)|$.

\begin{definition}
We say that an admissible function $\hat{x}$ is a weak local minimizer
(respectively weak local maximizer) to problem \eqref{problem}--\eqref{punkty}
if there exists $\delta >0$ such that
$\mathcal{L}(\hat{x})\leqslant \mathcal{L}(x)$
(respectively $\mathcal{L}(\hat{x})\geqslant \mathcal{L}(x)$)
for all admissible functions $x\in C_{k,n}^{1}$ satisfying the inequality
$||x-\hat{x}||_{1,\infty}<\delta$.
\end{definition}

For simplicity, we introduce the operators $[\cdot]^{\Delta}$ and $[\cdot]^{\nabla}$ by
$[x]^{\Delta}(t)=(t, x^{\sigma}(t),x^{\Delta}(t))$
and $[x]^{\nabla}(t)=(t,x^{\rho}(t),x^{\nabla}(t))$.
Along the text, $c$ denotes constants that are generic
and may change at each occurrence.

% -------------------------------------

\subsection{Euler--Lagrange equations}

Depending on the given boundary conditions, we can distinguish four different problems.
The first is problem $(P_{ab})$, where the two boundary conditions are specified.
To solve this problem we need a type of Euler--Lagrange necessary optimality condition.
This is given by Theorem~\ref{main} below.
Next two problems --- denoted by $(P_{a})$ and $(P_{b})$ --- occur when
$x(a)$ is given and $x(b)$ is free (problem $(P_{a})$) and when
$x(a)$ is free and $x(b)$ is specified (problem $(P_{b})$).
To solve both of them we need to use an Euler--Lagrange equation
and one transversality condition. The last problem --- denoted by  $(P)$ ---
occurs when both boundary conditions are not specified.
To find a solution for such a problem we need to use an Euler--Lagrange equation
and two transversality conditions (one at each time $a$ and $b$).
Transversality conditions are the subject of Section~\ref{sec:nbc}.

\begin{theorem}[Euler--Lagrange equations in integral form]
\label{main}
If $\hat{x}$ is a weak local solution
to problem \eqref{problem}--\eqref{punkty},
then the Euler--Lagrange equations\footnote{For brevity,
we are omitting the arguments of $H_{i}^{'}$, i.e.,
$H_{i}^{'}:=H_{i}^{'}\left(\mathcal{F}_{1}(\hat{x}),
\ldots,\mathcal{F}_{k+n}(\hat{x})\right)$,
where $\mathcal{F}_{i}(\hat{x})
=\int\limits_{a}^{b} f_{i}[\hat{x}]^\Delta(t)\Delta t$,
$i=1,\ldots,k$, and $\mathcal{F}_{i}(\hat{x})
=\int\limits_{a}^{b}f_{i}[\hat{x}]^{\nabla}(t)\nabla t$,
$i=k+1,\ldots,k+n$.}
\begin{multline}
\label{main_dla_nabla}
\sum\limits_{i=1}^{k}H_{i}^{'}\cdot
\left(f_{iv}[\hat{x}]^{\Delta}(\rho(t))-\int\limits_{a}^{\rho(t)}
f_{iy}[\hat{x}]^{\Delta}(\tau)\Delta \tau \right)\\
+\sum\limits_{i=k+1}^{k+n}H_{i}^{'}\cdot
\left(f_{iv}[\hat{x}]^{\nabla}(t)-\int\limits_{a}^{t}
f_{iy}[\hat{x}]^{\nabla}(\tau)\nabla \tau \right) = c,
\quad  t\in\mathbb{T}_{\kappa},
\end{multline}
and
\begin{multline}
\label{main_dla_delta}
\sum\limits_{i=1}^{k}H_{i}^{'}\cdot
\left(f_{iv}[\hat{x}]^{\Delta}(t)-\int\limits_{a}^{t}
f_{iy}[\hat{x}]^{\Delta}(\tau)\Delta \tau\right)\\
+\sum\limits_{i=k+1}^{k+n}H_{i}^{'}\cdot
\left(f_{iv}[\hat{x}]^{\nabla}(\sigma(t))-\int\limits_{a}^{\sigma(t)}
f_{iy}[\hat{x}]^{\nabla}(\tau)\nabla \tau\right) = c,
\quad  t\in\mathbb{T}^{\kappa},
\end{multline}
hold.
\end{theorem}

\begin{proof}
Suppose that $\mathcal{L}\left(x\right)$ has a weak local extremum at $\hat{x}$.
Consider a variation $h\in C_{k,n}^{1}$ of $\hat{x}$ for which we define the function
$\phi:\mathbb{R}\rightarrow\mathbb{R}$ by
$\phi(\varepsilon)=\mathcal{L}\left(\hat{x}+\varepsilon h\right)$.
A necessary condition for $\hat{x}$ to be an extremizer for
$\mathcal{L}\left(x\right)$ is given by
$\phi^{'}\left(\varepsilon\right)=0$ for $\varepsilon =0$.
Using the chain rule, we obtain that
\begin{multline*}
0=\phi^{'}\left(0\right)
=\sum\limits_{i=1}^{k}H_{i}^{'}\int\limits_{a}^{b}
\left(f_{iy}[\hat{x}]^{\Delta}(t)h^{\sigma}(t)
+f_{iv}[\hat{x}]^{\Delta}(t)h^{\Delta}(t)\right)\Delta t\\
+\sum\limits_{i=k+1}^{k+n}H_{i}^{'}\int\limits_{a}^{b}
\left(f_{iy}[\hat{x}]^{\nabla}(t)h^{\rho}(t)
+f_{iv}[\hat{x}]^{\nabla}(t)h^{\nabla}(t)\right)\nabla t.
\end{multline*}
Integration by parts of the first terms of both integrals gives
$$
\int\limits_{a}^{b} f_{iy}[\hat{x}]^{\Delta}(t)h^{\sigma}(t)\Delta t
=\left.\int\limits_{a}^{t}f_{iy}[\hat{x}]^{\Delta}(\tau)\Delta\tau h(t)\right|^{b}_{a}
- \int\limits_{a}^{b}\left(\int\limits_{a}^{t}
f_{iy}[\hat{x}]^{\Delta}(\tau)\Delta \tau\right) h^{\Delta}(t)\Delta t,
$$
$$
\int\limits_{a}^{b} f_{iy}[\hat{x}]^{\nabla}(t)h^{\rho}(t)\nabla t
=\left.\int\limits_{a}^{t}
f_{iy}[\hat{x}]^{\nabla}(\tau)\nabla\tau h(t)\right|^{b}_{a}
-\int\limits_{a}^{b}\left(\int\limits_{a}^{t}
f_{iy}[\hat{x}]^{\nabla}(\tau)\nabla \tau\right) h^{\nabla}(t)\nabla t.
$$
Thus, the necessary condition $\phi^{'}(0)=0$ can be written as
\begin{multline*}
\sum\limits_{i=1}^{k}H_{i}^{'}
\left[\left.\int\limits_{a}^{t}
f_{iy}[\hat{x}]^{\Delta}(\tau)\Delta\tau h(t)\right|^{b}_{a}
-\int\limits_{a}^{b}\left(\int\limits_{a}^{t}
f_{iy}[\hat{x}]^{\Delta}(\tau)\Delta \tau\right) h^{\Delta}(t)\Delta t\right.\\
+\left. \int\limits_{a}^{b}f_{iv}[\hat{x}]^{\Delta}(t)h^{\Delta}(t)\Delta t\right]
\end{multline*}
\begin{multline}
\label{glowne}
+\sum\limits_{i=k+1}^{k+n}H_{i}^{'}
\left[\left.\int\limits_{a}^{t}
f_{iy}[\hat{x}]^{\nabla}(\tau)\nabla\tau h(t)\right|^{b}_{a}
-\int\limits_{a}^{b}\left(\int\limits_{a}^{t}
f_{iy}[\hat{x}]^{\nabla}(\tau)\nabla \tau\right) h^{\nabla}(t)\nabla t\right.\\
+ \left. \int\limits_{a}^{b} f_{iv}[\hat{x}]^{\nabla}(t)h^{\nabla}(t)\nabla t\right] = 0.
\end{multline}
In particular, condition \eqref{glowne} holds for all variations
that are zero at both ends: $h(a)=h(b)=0$. Then, we obtain:
\begin{multline*}
\int\limits_{a}^{b}\sum\limits_{i=1}^{k}H_{i}^{'}h^{\Delta}(t)
\left(f_{iv}[\hat{x}]^{\Delta}(t)-\int\limits_{a}^{t}
f_{iy}[\hat{x}]^{\Delta}(\tau)\Delta \tau \right)\Delta t\\
+\int\limits_{a}^{b} \sum\limits_{i=k+1}^{k+n}H_{i}^{'}h^{\nabla}(t)
\left(f_{iv}[\hat{x}]^{\nabla}(t)-\int\limits_{a}^{t}
f_{iy}[\hat{x}]^{\nabla}(\tau)\nabla \tau \right)\nabla t = 0.
\end{multline*}
Introducing $\xi$ and $\chi$ by
\begin{equation}
\label{eq:st}
\xi(t):=\sum\limits_{i=1}^{k}H_{i}^{'}
\left(f_{iv}[\hat{x}]^{\Delta}(t)-\int\limits_{a}^{t}
f_{iy}[\hat{x}]^{\Delta}(\tau)\Delta \tau \right)
\end{equation}
and
\begin{equation}
\label{eq:rt}
\chi(t):=\sum\limits_{i=k+1}^{k+n}H_{i}^{'}
\left(f_{iv}[\hat{x}]^{\nabla}(t)-\int\limits_{a}^{t}
f_{iy}[\hat{x}]^{\nabla}(\tau)\nabla \tau \right),
\end{equation}
we then obtain the following relation:
\begin{equation}
\label{skrot}
\int\limits_{a}^{b} h^{\Delta}(t)\xi(t)\Delta t
+\int\limits_{a}^{b} h^{\nabla}(t)\chi(t)\nabla t = 0.
\end{equation}
We consider two cases. (i) Firstly,
we change the first integral of \eqref{skrot}
and we obtain two nabla-integrals and, subsequently,
the equation \eqref{main_dla_nabla}. (ii) In the second case,
we change the second integral of \eqref{skrot}
to obtain two delta-integrals,
which lead us to \eqref{main_dla_delta}.

(i) Using relation \eqref{polaczenie calka delta}
of Theorem~\ref{polaczenie calka delta i nabla}, we obtain:
$$
\int\limits_{a}^{b}\left(h^{\Delta}(t)\right)^{\rho} \xi^{\rho}(t)\nabla t
+\int\limits_{a}^{b} h^{\nabla} (t)\chi(t)\nabla t = 0.
$$
Using \eqref{polaczenie rozniczka nabla}
of Theorem~\ref{polaczenie rozniczka delta i nabla} we have
\begin{equation*}
\int\limits_{a}^{b}h^{\nabla}(t)\left(\xi^{\rho}(t)+\chi(t)\right)\nabla t = 0.
\end{equation*}
By the Dubois--Reymond Lemma~\ref{Dubois-Reymond nabla}
\begin{equation}
\label{skrot z nabla}
\xi^{\rho}(t)+\chi(t)=const
\end{equation}
and we obtain \eqref{main_dla_nabla}.

(ii) From \eqref{skrot}, and using relation \eqref{polaczenie calka nabla}
of Theorem~\ref{polaczenie calka delta i nabla},
$$
\int\limits_{a}^{b}h^{\Delta}(t)\xi(t)\Delta t
+\int\limits_{a}^{b}(h^{\nabla}(t))^{\sigma} \chi^{\sigma}(t)\Delta t = 0.
$$
Using \eqref{polaczenie rozniczka delta}
of Theorem~\ref{polaczenie rozniczka delta i nabla}, we get:
$\int\limits_{a}^{b}h^{\Delta}(t)(\xi(t) +\chi^{\sigma}(t))\Delta t = 0$.
From the Dubois--Reymond Lemma~\ref{Dubois-Reymond delta}, it follows that
$\xi(t)+ \chi^{\sigma}(t)=const$. Hence, we obtain the Euler--Lagrange
equation \eqref{main_dla_delta}.
\end{proof}

A time scale $\mathbb{T}$ is said to be regular if the following two conditions
are satisfied simultaneously for all $t\in\mathbb{T}$:
$\sigma(\rho(t))=t$ and $\rho(\sigma(t))=t$. For regular time scales,
the Euler--Lagrange equations \eqref{main_dla_nabla} and \eqref{main_dla_delta}
coincide; on a general time scale, they are different.
Such a difference is illustrated in Example~\ref{ex:1}.

\begin{example}
\label{ex:1}
Let us consider the irregular time scale
$\mathbb{T}=\mathbb{P}_{1,1}=\bigcup\limits_{k=0}^{\infty}\left[2k,2k+1\right]$.
We show that for this time scale there is a difference between
the Euler--Lagrange equations \eqref{main_dla_nabla} and \eqref{main_dla_delta}.
The forward and backward jump operators  are given by
$$
\sigma(t)=
\begin{cases}
t,\quad t\in\bigcup\limits_{k=0}^{\infty}[\left.2k,2k+1)\right.
\\t+1, \quad t\in \bigcup\limits_{k=0}^{\infty}\left\{2k+1\right\},
\end{cases}
\quad
\rho(t)=
\begin{cases}
t,\quad t\in\bigcup\limits_{k=0}^{\infty}(\left.2k,2k+1]\right.\\
t-1, \quad t\in \bigcup\limits_{k=1}^{\infty}\left\{2k\right\}\\
0, \quad t = 0.
\end{cases}
$$
For $t = 0$ and $t\in \bigcup\limits_{k=0}^{\infty}\left(2k,2k+1\right)$,
equations \eqref{main_dla_nabla} and \eqref{main_dla_delta} coincide.
We can distinguish between them for
$t\in \bigcup\limits_{k=0}^{\infty}\left\{2k+1\right\}$
and $t\in \bigcup\limits_{k=1}^{\infty}\left\{2k\right\}$.
In what follows we use the notations \eqref{eq:st} and \eqref{eq:rt}.
If $t\in \bigcup\limits_{k=0}^{\infty}\left\{2k+1\right\}$,
then we obtain from \eqref{main_dla_nabla} and \eqref{main_dla_delta}
the Euler--Lagrange equations
$\xi(t) + \chi(t) = c$ and $\xi(t) + \chi(t+1) = c$, respectively.
If $t\in \bigcup\limits_{k=1}^{\infty}\left\{2k\right\}$,
then the Euler--Lagrange equation \eqref{main_dla_nabla}
has the form $\xi(t-1) + \chi(t) = c$
while \eqref{main_dla_delta} takes the form $\xi(t) + \chi(t) = c$.
\end{example}

% -------------------------------------

\subsection{Natural boundary conditions}
\label{sec:nbc}

In this section we consider the situation when we want to minimize or maximize
the variational functional \eqref{problem},
but boundary conditions $x(a)$ and/or $x(b)$ are free.

\begin{theorem}[Transversality condition at the initial time $t = a$]
\label{punkt pocz}
Let $\mathbb{T}$ be a time scale for which $\rho(\sigma(a))=a$.
If $\hat{x}$ is a weak local solution
to \eqref{problem} with $x(a)$ not specified, then
\begin{equation}
\label{pocz dla nabla}
\sum\limits_{i=1}^{k}H_{i}^{'}
\cdot
f_{iv}[\hat{x}]^{\Delta}(a)
+\sum\limits_{i=k+1}^{k+n}H_{i}^{'}\cdot
\left(
f_{iv}[\hat{x}]^{\nabla}(\sigma(a))
- \int\limits^{\sigma(a)}_{a}f_{iy}[\hat{x}]^{\nabla}(t)\nabla t
\right) = 0
\end{equation}
holds together with the Euler--Lagrange equations
\eqref{main_dla_nabla} and \eqref{main_dla_delta}.
\end{theorem}

\begin{proof}
From \eqref{glowne} and \eqref{skrot z nabla} we have
$$
\sum\limits_{i=1}^{k}H_{i}^{'}
\left.\int\limits_{a}^{t}
f_{iy}[\hat{x}]^{\Delta}(\tau)\Delta \tau h(t)\right|^{b}_{a}
+\sum\limits_{i=k+1}^{k+n}H_{i}^{'}
\left.\int\limits_{a}^{t}f_{iy}[\hat{x}]^{\nabla}(\tau)\nabla \tau h(t)\right|^{b}_{a}
+\int\limits_{a}^{b}h^{\nabla}(t)\cdot c\nabla t = 0.
$$
Therefore,
$$
\sum\limits_{i=1}^{k}H_{i}^{'}
\left.\int\limits_{a}^{t}f_{iy}[\hat{x}]^{\Delta}(\tau)\Delta \tau h(t)\right|^{b}_{a}
+\sum\limits_{i=k+1}^{k+n}H_{i}^{'}
\left.\int\limits_{a}^{t}f_{iy}[\hat{x}]^{\nabla}(\tau)\nabla \tau h(t)\right|^{b}_{a}
+\left.h(t)\cdot c\right|_{a}^{b} = 0.
$$
Next, we deduce that
\begin{multline}
\label{glowne z nabla}
h(b)
\left[\sum\limits_{i=1}^{k}H_{i}^{'}
\int\limits_{a}^{b}f_{iy}[\hat{x}]^{\Delta}(\tau)\Delta\tau
+\sum\limits_{i=k+1}^{k+n}H_{i}^{'}
\int\limits_{a}^{b}f_{iy}[\hat{x}]^{\nabla}(\tau)\nabla \tau+c\right]\\
-h(a)
\left[\sum\limits_{i=1}^{k}H_{i}^{'}
\int\limits_{a}^{a}f_{iy}[\hat{x}]^{\Delta}(\tau)\Delta\tau
+\sum\limits_{i=k+1}^{k+n}H_{i}^{'}
\int\limits_{a}^{a}f_{iy}[\hat{x}]^{\nabla}(\tau)\nabla \tau+c\right] = 0,
\end{multline}
where
\begin{equation}
\label{E-L skrot nabla}
c=\xi(\rho(t))+\chi(t).
\end{equation}
The Euler--Lagrange equation \eqref{main_dla_nabla}
of Theorem~\ref{main} (or equation \eqref{E-L skrot nabla})
is given at $t=\sigma(a)$ as
\begin{multline*}
\sum\limits_{i=1}^{k}H_{i}^{'}
\left(f_{iv}[\hat{x}]^{\Delta}(\rho(\sigma(a)))
-\int\limits_{a}^{\rho(\sigma(a))}f_{iy}[\hat{x}]^{\Delta}(\tau)\Delta \tau\right)\\
+\sum\limits_{i=k+1}^{k+n}H_{i}^{'}
\left(f_{iv}[\hat{x}]^{\nabla}(\sigma(a))-\int\limits_{a}^{\sigma(a)}
f_{iy}[\hat{x}]^{\nabla}(\tau)\nabla \tau\right) = c.
\end{multline*}
We conclude that
$$
\sum\limits_{i=1}^{k}H_{i}^{'}\cdot
f_{iv}[\hat{x}]^{\Delta}(a)
+\sum\limits_{i=k+1}^{k+n} H_{i}^{'}\cdot
\left(f_{iv}[\hat{x}]^{\nabla}(\sigma(a)) - \int\limits_{a}^{\sigma(a)}
f_{iy}[\hat{x}]^{\nabla}(\tau)\nabla \tau\right) = c.
$$
Restricting the variations $h$ to those such that $h(b)=0$,
it follows from \eqref{glowne z nabla} that $h(a)\cdot c=0$.
From the arbitrariness of $h$, we conclude that $c = 0$.
Hence, we obtain \eqref{pocz dla nabla}.
\end{proof}

\begin{theorem}[Transversality condition at the terminal time $t = b$]
\label{punkt kon}
Let $\mathbb{T}$ be a time scale for which $\sigma(\rho(b))=b$.
If $\hat{x}$ is a weak local solution to \eqref{problem}
with $x(b)$ not specified, then
\begin{equation}
\label{kon dla delta}
\sum\limits_{i=1}^{k}H_{i}^{'}
\left(
f_{iv}[\hat{x}]^{\Delta}(\rho(b))
+
\int\limits_{\rho(b)}^{b}f_{iy}[\hat{x}]^{\Delta}(t)\Delta t
\right)
+\sum\limits_{i=k+1}^{k+n}H_{i}^{'}
\cdot
f_{iv}[\hat{x}]^{\nabla}(b) = 0
\end{equation}
holds together with the Euler--Lagrange equations
\eqref{main_dla_nabla} and \eqref{main_dla_delta}.
\end{theorem}

\begin{proof}
The calculations in the proof of Theorem~\ref{punkt pocz} give us
\eqref{glowne z nabla}. When $h(a)=0$,
the Euler--Lagrange equation \eqref{main_dla_delta}
of Theorem~\ref{main} has the following form at $t=\rho(b)$:
\begin{multline*}
\sum\limits_{i=1}^{k}H_{i}^{'}
\left(
f_{iv}[\hat{x}]^{\Delta}(\rho(b))-\int\limits_{a}^{\rho(b)}
f_{iy}[\hat{x}]^{\Delta}(\tau)\Delta \tau\right)\\
+\sum\limits_{i=k+1}^{k+n}H_{i}^{'}
\left(f_{iv}[\hat{x}]^{\nabla}(\sigma(\rho(b)))-\int\limits_{a}^{\sigma(\rho(b))}
f_{iy}[\hat{x}]^{\nabla}(t)\nabla\tau\right) = c.
\end{multline*}
Then,
\begin{multline}
\label{E-L trans}
\sum\limits_{i=1}^{k}H_{i}^{'}
\left(
f_{iv}[\hat{x}]^{\Delta}(\rho(b))-\int\limits_{a}^{\rho(b)}
f_{iy}[\hat{x}]^{\Delta}(\tau)\Delta \tau\right)\\
+\sum\limits_{i=k+1}^{k+n}H_{i}^{'}
\left(f_{iv}[\hat{x}]^{\nabla}(b)-\int\limits_{a}^{b}
f_{iy}[\hat{x}]^{\nabla}(t)\nabla\tau\right) = c.
\end{multline}
We obtain \eqref{kon dla delta} from
\eqref{glowne z nabla} and \eqref{E-L trans}.
\end{proof}

Several new interesting results can be immediately obtained
from Theorems~\ref{main}, \ref{punkt pocz} and \ref{punkt kon}.
An example of such results is given by Corollary~\ref{cor iloraz}.

\begin{corollary}
\label{cor iloraz}
If $\hat{x}$ is a solution to the problem
\begin{gather*}
\mathcal{L}(x)
=\frac{\int\limits_{a}^{b}f_{1}(t,x^{\sigma}(t),x^{\Delta}(t))
\Delta t}{\int\limits_{a}^{b}f_{2}(t,x^{\rho}(t),x^{\nabla}(t))\nabla t}
\longrightarrow \textrm{extr},\\
(x(a)=x_{a}), \quad (x(b)=x_{b}),
\end{gather*}
then the Euler--Lagrange equations
$$
\frac{1}{\mathcal{F}_{2}}
\left(
f_{1v}[\hat{x}]^{\Delta}(\rho(t))-\int\limits_{a}^{\rho(t)}
f_{1y}[\hat{x}]^{\Delta}(\tau)\Delta \tau\right)
- \frac
{\mathcal{F}_{1}}
{\mathcal{F}_{2}^{2}}
\left(f_{2v}[\hat{x}]^{\nabla}(t)-\int\limits_{a}^{t}
f_{2y}[\hat{x}]^{\nabla}(\tau)\nabla \tau\right) = c
$$
and
$$
\frac{1}{\mathcal{F}_{2}}
\left(
f_{1v}[\hat{x}]^{\Delta}(t)-\int\limits_{a}^{t}
f_{1y}[\hat{x}]^{\Delta}(\tau)\Delta \tau\right)
-\frac{\mathcal{F}_{1}}{\mathcal{F}_{2}^{2}}
\left(f_{2v}[\hat{x}]^{\nabla}(\sigma(t))
-\int\limits^{\sigma (t)}_{a}
f_{2y}[\hat{x}]^{\nabla}(\tau)\nabla \tau\right)=c
$$
hold for all $t\in [a,b]_{\kappa}^{\kappa}$, where
$$
\mathcal{F}_{1}:={\int\limits_{a}^{b}
f_{1}(t,\hat{x}^{\sigma}(t),\hat{x}^{\Delta}(t))\Delta t}
\quad \text{ and } \quad
\mathcal{F}_{2}:={\int\limits_{a}^{b}
f_{2}(t,\hat{x}^{\rho}(t),\hat{x}^{\nabla}(t))\nabla t}.
$$
Moreover, if $x(a)$ is free and $\rho(\sigma(a))=a$, then
$$
\frac{1}{\mathcal{F}_{2}}
f_{1v}[\hat{x}]^{\Delta}(a)
-\frac{\mathcal{F}_{1}}{\mathcal{F}_{2}^{2}}
\left(f_{2v}[\hat{x}]^{\nabla}(\sigma(a))-\int\limits_{a}^{\sigma(a)}
f_{2y}[\hat{x}]^{\nabla}(t)\nabla t\right) =0;
$$
if $x(b)$ is free and $\sigma(\rho(b))=b$, then
$$
\frac{1}{\mathcal{F}_{2}}
\left(f_{1v}[\hat{x}]^{\Delta}(\rho(b))+\int\limits^{b}_{\rho(b)}
f_{1y}[\hat{x}]^{\Delta}(t)\Delta t\right)
-\frac{\mathcal{F}_{1}}{\mathcal{F}_{2}^{2}}
f_{2v}[\hat{x}]^{\nabla}(b)=0.
$$
\end{corollary}

% -------------------------------------

\subsection{Isoperimetric problems}
\label{sub:sec:iso:p}

Let us consider the general composition
isoperimetric problem on time scales
subject to given boundary conditions.
The problem consists of minimizing or maximizing
\begin{multline}
\label{problem iso}
\mathcal L(x)=H\left(\int\limits_{a}^{b}f_{1}(t,x^{\sigma}(t),x^{\Delta}(t))\Delta t,
\ldots,\int\limits_{a}^{b}f_{k}(t,x^{\sigma}(t),x^{\Delta}(t))\Delta t,\right.\\
\left.\int\limits_{a}^{b}f_{k+1}(t,x^{\rho}(t),x^{\nabla}(t))\nabla t,\ldots,
\int\limits_{a}^{b}f_{k+n}(t,x^{\rho}(t),x^{\nabla}(t)) \nabla t \right)
\end{multline}
in the class of functions $x\in C_{k,n}^{1}$ satisfying the boundary conditions
\begin{equation}
\label{warunki iso}
x(a)=x_{a},\quad x(b)=x_{b},
\end{equation}
and the generalized isoperimetric constraint
\begin{multline}
\label{constraint}
\mathcal{K}(x)
=P\left(\int\limits_{a}^{b}g_{1}(t,x^{\sigma}(t),x^{\Delta}(t))\Delta t,
\ldots,\int\limits_{a}^{b}g_{m}(t,x^{\sigma}(t),x^{\Delta}(t))\Delta t,\right.\\
\left.\int\limits_{a}^{b}g_{m+1}(t,x^{\rho}(t),x^{\nabla}(t))\nabla t,\ldots,
\int\limits_{a}^{b}g_{m+p}(t,x^{\rho}(t),x^{\nabla}(t)) \nabla t \right)=d,
\end{multline}
where $x_{a},x_{b},d\in\mathbb{R}$. We assume that:
\begin{enumerate}

\item
the functions $H:\mathbb{R}^{n+k}\rightarrow\mathbb{R}$
and $P:\mathbb{R}^{m+p}\rightarrow\mathbb{R}$
have continuous partial derivatives with respect to all their arguments,
which we denote by $H_{i}^{'}$, $i=1,\ldots,n+k$,
and $P_{i}^{'}$, $i=1,\ldots,m+p$;

\item
functions $(t,y,v)\rightarrow f_{i}(t,y,v)$,
$i=1,\ldots, n+k$, and
$(t,y,v)\rightarrow g_{j}(t,y,v)$, $j=1,\ldots,m+p$,
from $[a,b]\times\mathbb{R}^{2}$ to $\mathbb{R}$,
have partial continuous derivatives with respect to $y$ and $v$ for all
$t\in [a,b]$, which we denote by $f_{iy}$, $f_{iv}$, and $g_{jy}, g_{jv}$;

\item for all $x\in C_{k+m,n+p}^{1}$,
$f_{i}$, $f_{iy}$, $f_{iv}$ and $g_{j},g_{jy}$, $g_{jv}$
are continuous in $t\in [a,b]^{\kappa}$,
$i=1,\ldots,k$, $j=1,\ldots,m$,
and continuous in $t\in [a,b]_{\kappa}$,
$i=k+1,\ldots,k+n$, $j=m+1,\ldots,m+p$.
\end{enumerate}

\begin{definition}
We say that an admissible function $\hat{x}$ is a weak local minimizer
(respectively a weak local maximizer) to the isoperimetric problem
\eqref{problem iso}--\eqref{constraint}, if there exists a $\delta >0$
such that $\mathcal{L}(\hat{x})\leqslant \mathcal{L}(x)$
(respectively $\mathcal{L}(\hat{x})\geqslant \mathcal{L}(x)$)
for all admissible functions $x\in C_{k+m,n+p}^{1}$
satisfying the boundary conditions \eqref{warunki iso},
the isoperimetric constraint \eqref{constraint},
and inequality $||x-\hat{x}||_{1,\infty}<\delta$.
\end{definition}

Let us define $u$ and $w$ by
\begin{equation}
\label{eq:def:ut}
u(t):=
\sum\limits_{i=1}^{m}P_{i}^{'}
\left(g_{iv}[\hat{x}]^{\Delta}(t)-\int\limits_{a}^{t}
g_{iy}[\hat{x}]^{\Delta}(\tau)\Delta \tau \right)
\end{equation}
and
\begin{equation}
\label{eq:def:wt}
w(t):=
\sum\limits_{i=m+1}^{m+p}P_{i}^{'}
\left(g_{iv}[\hat{x}]^{\nabla}(t)-\int\limits_{a}^{t}
g_{iy}[\hat{x}]^{\nabla}(\tau)\nabla \tau \right),
\end{equation}
where we omit, for brevity, the argument of $P_{i}^{'}$:
$P_{i}^{'}:=P_{i}^{'}(\mathcal{G}_{1}(\hat{x}),
\ldots,\mathcal{G}_{m+p}(\hat{x}))$
with $\mathcal{G}_{i}(\hat{x})=\int\limits_{a}^{b}
g_{i}(t,\hat{x}^{\sigma}(t),\hat{x}^{\Delta}(t))\Delta t$,
$i=1,\ldots,m$, and
$\mathcal{G}_{i}(\hat{x})=\int\limits_{a}^{b}
g_{i}(t,\hat{x}^{\rho}(t),\hat{x}^{\nabla}(t))\nabla t$,
$i=m+1,\ldots,m+p$.

\begin{definition}
An admissible function $\hat{x}$ is said to be an extremal for $\mathcal{K}$ if
$u(t) + w(\sigma(t)) = const$ and $u(\rho(t)) + w(t) = const$ for all $t\in[a,b]_\kappa^\kappa$.
An extremizer (i.e., a weak local minimizer or a weak local maximizer)
to problem \eqref{problem iso}--\eqref{constraint} that is not an extremal for $\mathcal{K}$
is said to be a normal extremizer; otherwise (i.e., if it is an extremal for $\mathcal{K}$),
the extremizer is said to be abnormal.
\end{definition}

\begin{theorem}[Optimality condition to the
isoperimetric problem \eqref{problem iso}--\eqref{constraint}]
\label{twierdzenie iso}
Let $\xi$ and $\chi$ be given as in \eqref{eq:st}
and \eqref{eq:rt}, and $u$ and $w$ be given as in
\eqref{eq:def:ut} and \eqref{eq:def:wt}.
If $\hat{x}$ is a normal extremizer to the isoperimetric problem
\eqref{problem iso}--\eqref{constraint}, then there exists
a real number $\lambda$ such that
\begin{enumerate}
\item $\xi^{\rho}(t)+\chi(t)-\lambda\left(u^{\rho}(t)+w(t)\right)= const$;
\item $\xi(t)+\chi^{\sigma}(t)-\lambda\left(u^{\rho}(t)+w(t)\right)= const$;
\item $\xi^{\rho}(t)+\chi(t)-\lambda\left(u(t)+w^{\sigma}(t)\right)= const$;
\item $\xi(t)+\chi^{\sigma}(t)-\lambda\left(u(t)+w^{\sigma}(t)\right)= const$;
\end{enumerate}
for all $t\in [a,b]^{\kappa}_{\kappa}$.
\end{theorem}

\begin{proof}
We prove the first item of Theorem~\ref{twierdzenie iso}.
The other items are proved in a similar way.
Consider a variation of $\hat{x}$ such that
$\overline{x}=\hat{x}+\varepsilon_{1}h_{1}+\varepsilon_{2}h_{2}$,
where $h_{i}\in C^{1}_{k+m,n+p}$ and $h_{i}(a)=h_{i}(b)=0$, $i=1,2$,
and parameters $\varepsilon_{1}$ and $\varepsilon_{2}$ are such that
$||\overline{x}-\hat{x}||_{1,\infty}<\delta$ for some $\delta>0$.
Function $h_{1}$ is arbitrary and $h_{2}$
will be chosen later. Define
\begin{multline*}
\mathcal{\overline{K}}(\varepsilon_{1},\varepsilon_{2})=\mathcal{K}(\overline{x})
= P\left(\int\limits_{a}^{b}
g_{1}(t,\overline{x}^{\sigma}(t),\overline{x}^{\Delta}(t))\Delta t,
\ldots,\int\limits_{a}^{b}
g_{m}(t,\overline{x}^{\sigma}(t),\overline{x}^{\Delta}(t))\Delta t,\right.\\
\left.\int\limits_{a}^{b}
g_{m+1}(t,\overline{x}^{\rho}(t),\overline{x}^{\nabla}(t))\nabla t,\ldots,
\int\limits_{a}^{b}
g_{m+p}(t,\overline{x}^{\rho}(t),\overline{x}^{\nabla}(t))\nabla t\right) - d.
\end{multline*}
A direct calculation gives
\begin{multline*}
\left.\frac{\partial\mathcal{\overline{K}}}{\partial\varepsilon_{2}}\right|_{(0,0)}
=\sum\limits_{i=1}^{m}P_{i}^{'}\int\limits_{a}^{b}
\left(g_{iy}[\hat{x}]^{\Delta}(t)h^{\sigma}_{2}(t)
+g_{iv}[\hat{x}]^{\Delta}(t)h^{\Delta}_{2}(t)\right)\Delta t\\
+\sum\limits_{i=m+1}^{m+p}P_{i}^{'}\int\limits_{a}^{b}
\left(g_{iy}[\hat{x}]^{\nabla}(t)h^{\rho}_{2}(t)
+g_{iv}[\hat{x}]^{\nabla}(t)h^{\nabla}_{2}(t)\right)\nabla t.
\end{multline*}
Integration by parts of the first terms of both integrals gives:
\begin{multline*}
\sum\limits_{i=1}^{m}P_{i}^{'}
\left[\left.\int\limits_{a}^{t}g_{iy}[\hat{x}]^{\Delta}(\tau)\Delta\tau h_{2}(t)\right|^{b}_{a}
-\int\limits_{a}^{b}\left(\int\limits_{a}^{t} g_{iy}[\hat{x}]^{\Delta}(\tau)\Delta \tau\right)
h_{2}^{\Delta}(t)\Delta t\right.\\
+\left. \int\limits_{a}^{b}g_{iv}[\hat{x}]^{\Delta}(t)h_{2}^{\Delta}(t)\Delta t\right]\\
+\sum\limits_{i=m+1}^{m+p}P_{i}^{'} \left[\left.\int\limits_{a}^{t}
g_{iy}[\hat{x}]^{\nabla}(\tau)\nabla\tau h_{2}(t)\right|^{b}_{a}
-\int\limits_{a}^{b}\left(\int\limits_{a}^{t}
g_{iy}[\hat{x}]^{\nabla}(\tau)\nabla \tau\right) h_{2}^{\nabla}(t)\nabla t\right.\\
+ \left. \int\limits_{a}^{b} g_{iv}[\hat{x}]^{\nabla}(t)h_{2}^{\nabla}(t)\nabla t\right].
\end{multline*}
Since $h_{2}(a)=h_{2}(b)=0$, then
\begin{multline*}
\int\limits_{a}^{b}
\sum\limits_{i=1}^{m}P_{i}^{'}h_{2}^{\Delta}(t)
\left(g_{iv}[\hat{x}]^{\Delta}(t)-\int\limits_{a}^{t}
g_{iy}[\hat{x}]^{\Delta}(\tau)\Delta \tau \right)\Delta t\\
+\int\limits_{a}^{b}
\sum\limits_{i=m+1}^{m+p}P_{i}^{'}h_{2}^{\nabla}(t)
\left(g_{iv}[\hat{x}]^{\nabla}(t)-\int\limits_{a}^{t}
g_{iy}[\hat{x}]^{\nabla}(\tau)\nabla \tau \right)\nabla t.
\end{multline*}
Therefore,
$$
\left.\frac{\partial\mathcal{\overline{K}}}{\partial\varepsilon_{2}}\right|_{(0,0)}
= \int\limits_{a}^{b}h_{2}^{\Delta}(t)u(t)\Delta t
+\int\limits_{a}^{b}h_{2}^{\nabla}(t)w(t)\nabla t.
$$
Using relation \eqref{polaczenie rozniczka nabla}
of Theorem~\ref{polaczenie rozniczka delta i nabla},
we obtain that
\begin{equation*}
\int\limits_{a}^{b}\left(h_{2}^{\Delta}\right)^{\rho}(t)u^{\rho}(t)\nabla t
+\int\limits_{a}^{b}h_{2}^{\nabla}(t) w(t)\nabla t
=\int\limits_{a}^{b}h_{2}^{\nabla}(t)\left(u^{\rho}(t)+w(t)\right)\nabla t.
\end{equation*}
By the Dubois--Reymond Lemma~\ref{Dubois-Reymond nabla},
there exists a function $h_{2}$ such that
$\left.\frac{\partial\mathcal{\overline{K}}}{\partial\varepsilon_{2}}\right|_{(0,0)}\neq 0$.
Since $\mathcal{\overline{K}}(0,0)=0$, there exists a function $\varepsilon_{2}$,
defined in the neighborhood of zero, such that
$\mathcal{\overline{K}}(\varepsilon_{1}, \varepsilon_{2}(\varepsilon_{1}))=0$,
i.e., we may choose a subset of variations $\hat{x}$ satisfying the isoperimetric constraint.
Let us consider the real function
\begin{multline*}
\mathcal{\overline{L}}(\varepsilon_{1},\varepsilon_{2})=\mathcal{L}(\overline{x})
=H\left(\int\limits_{a}^{b}f_{1}(t,\overline{x}^{\sigma}(t),\overline{x}^{\Delta}(t))\Delta t,
\ldots,\int\limits_{a}^{b}f_{k}(t,\overline{x}^{\sigma}(t),\overline{x}^{\Delta}(t))\Delta t,\right.\\
\left.\int\limits_{a}^{b}f_{k+1}(t,\overline{x}^{\rho}(t),\overline{x}^{\nabla}(t))\nabla t,\ldots,
\int\limits_{a}^{b}f_{k+n}(t,\overline{x}^{\rho}(t),\overline{x}^{\nabla}(t)) \nabla t \right).
\end{multline*}
The point $(0,0)$ is an extremal of $\mathcal{\overline{L}}$ subject to the constraint
$\mathcal{\overline{K}}=0$ and $\nabla\mathcal{\overline{K}}(0,0)\neq 0$.
By the Lagrange multiplier rule, there exists $\lambda \in\mathbb{R}$ such that
$\nabla\left(\mathcal{\overline{L}}(0,0)-\lambda\mathcal{\overline{K}}(0,0)\right)=0$.
Because $h_{1}(a)=h_{2}(b)=0$, we have
\begin{multline*}
\left.\frac{\partial\mathcal{\overline{L}}}
{\partial\varepsilon_{1}}\right|_{(0,0)}
=\sum\limits_{i=1}^{k}H_{i}^{'}\int\limits_{a}^{b}
\left(f_{iy}[\hat{x}]^{\Delta}(t)h_{1}^{\sigma}(t)
+f_{iv}[\hat{x}]^{\Delta}(t)h_{1}^{\Delta}(t)\right)\Delta t\\
+\sum\limits_{i=k+1}^{k+n}H_{i}^{'}\int\limits_{a}^{b}
\left(f_{iy}[\hat{x}]^{\nabla}(t)h_{1}^{\rho}(t)
+f_{iv}[\hat{x}]^{\nabla}(t)h_{1}^{\nabla}(t)\right)\nabla t.
\end{multline*}
Integrating by parts, and using $h_{1}(a)=h_{1}(b)=0$, gives
$$
\left.\frac{\partial\mathcal{\overline{L}}}
{\partial\varepsilon_{1}}\right|_{(0,0)}
=\int\limits_{a}^{b}h_{1}^{\Delta}(t)\xi(t)\Delta t
+\int\limits_{a}^{b}h_{1}^{\nabla}(t) \chi(t)\nabla t.
$$
Using \eqref{polaczenie calka delta}
of Theorem~\ref{polaczenie calka delta i nabla}
and \eqref{polaczenie rozniczka nabla}
of Theorem~\ref{polaczenie rozniczka delta i nabla}, we obtain that
$$
\left.\frac{\partial\mathcal{\overline{L}}}
{\partial\varepsilon_{1}}\right|_{(0,0)}=
\int\limits_{a}^{b}\left(h_{1}^{\Delta}\right)^{\rho}(t)\xi^{\rho}(t)\nabla t
+\int\limits_{a}^{b}h_{1}^{\nabla}(t) \chi(t)\nabla t
= \int\limits_{a}^{b}h_{1}^{\nabla}(t)\left(\xi^{\rho}(t)+\chi(t)\right)\nabla t
$$
and
\begin{multline*}
\left.\frac{\partial\mathcal{\overline{K}}}{\partial\varepsilon_{1}}\right|_{(0,0)}
=\sum\limits_{i=1}^{m}P_{i}^{'}\int\limits_{a}^{b}
\left(g_{iy}[\hat{x}]^{\Delta}(t)h^{\sigma}_{1}(t)
+g_{iv}[\hat{x}]^{\Delta}(t)h^{\Delta}_{1}(t)\right)\Delta t\\
+\sum\limits_{i=m+1}^{m+p}P_{i}^{'}\int\limits_{a}^{b}
\left(g_{iy}[\hat{x}]^{\nabla}(t)h^{\rho}_{1}(t)
+g_{iv}[\hat{x}]^{\nabla}(t)h^{\nabla}_{1}(t)\right)\nabla t.
\end{multline*}
Integrating by parts, and recalling that $h_{1}(a)=h_{1}(b)=0$,
$$
\left.\frac{\partial\mathcal{\overline{K}}}{\partial\varepsilon_{1}}\right|_{(0,0)}
=\int\limits_{a}^{b}h_{1}^{\Delta}(t)u(t)\Delta t
+ \int\limits_{a}^{b}h_{1}^{\nabla}(t)w(t)\nabla t.
$$
Using relation \eqref{polaczenie calka delta}
of Theorem~\ref{polaczenie calka delta i nabla}
and relation \eqref{polaczenie rozniczka nabla}
of Theorem~\ref{polaczenie rozniczka delta i nabla},
we obtain that
\begin{multline*}
\left.\frac{\partial\mathcal{\overline{K}}}{\partial\varepsilon_{1}}\right|_{(0,0)}
= \int\limits_{a}^{b}\left(h_{1}^{\Delta}\right)^{\rho}(t)u^{\rho}(t)\nabla t
+\int\limits_{a}^{b}h_{1}^{\nabla}(t) w(t)\nabla t
=\int\limits_{a}^{b}h_{1}^{\nabla}(t)\left(u^{\rho}(t)+w(t)\right)\nabla t.
\end{multline*}
Since
$\left.\frac{\partial\mathcal{\overline{L}}}
{\partial\varepsilon_{1}}\right|_{(0,0)}
-\lambda
\left.\frac{\partial\mathcal{\overline{K}}}
{\partial\varepsilon_{1}}\right|_{(0,0)}=0$, then
$\int\limits_{a}^{b}h_{1}^{\nabla}(t)
\left[
\xi^{\rho}(t)+\chi(t)-\lambda\left(u^{\rho}(t)+w(t)\right)
\right]
\nabla t=0$
for any $h_{1}\in C_{k+m,n+p}$.
Therefore, by the Dubois--Reymond Lemma~\ref{Dubois-Reymond nabla},
one has $\xi^{\rho}(t)+\chi(t)-\lambda\left(u^{\rho}(t)+w(t)\right)=c$,
where $c\in\mathbb{R}$.
\end{proof}

\begin{remark}
One can easily cover both normal and abnormal extremizers
with Theorem~\ref{twierdzenie iso},
if in the proof we use the abnormal
Lagrange multiplier rule \cite{book:vanBrunt}.
\end{remark}

% -------------------------------------

\section{Illustrative examples}
\label{sec:examples}

We begin with a non-autonomous problem.

\begin{example}
\label{ex:1:se}
Consider the problem
\begin{equation}
\label{iloraz}
\begin{gathered}
\mathcal{L}(x)=
\frac{\int\limits_{0}^{1} t x^{\Delta}(t) \Delta t}
{\int\limits_{0}^{1}(x^{\nabla}(t))^{2}\nabla t}
\longrightarrow \min, \\
x(0)=0, \quad x(1)=1.
\end{gathered}
\end{equation}
If $x$ is a local minimizer to problem \eqref{iloraz},
then the Euler--Lagrange equations of Corollary~\ref{cor iloraz} must hold, i.e.,
$$
\frac{1}{\mathcal{F}_{2}}\rho(t)-2\frac{\mathcal{F}_{1}}{\mathcal{F}_{2}^{2}}
x^{\nabla}(t)=c \quad \text{ and } \quad
\frac{1}{\mathcal{F}_{2}}t-2\frac{\mathcal{F}_{1}}{\mathcal{F}_{2}^{2}}
x^{\nabla}(\sigma(t))=c,
$$
where
$\mathcal{F}_{1}:=\mathcal{F}_{1}(x)=\int\limits_{0}^{1}t x^{\Delta}(t)\Delta t$
and $\mathcal{F}_{2}:=\mathcal{F}_{2}(x)=\int\limits_{0}^{1}(x^{\nabla}(t))^{2}\nabla t$.
Let us consider the second equation. Using \eqref{polaczenie rozniczka delta}
of Theorem~\ref{polaczenie rozniczka delta i nabla}, it can be written as
\begin{equation}
\label{1 row iloraz}
\frac{1}{\mathcal{F}_{2}}t-2\frac{\mathcal{F}_{1}}{\mathcal{F}_{2}^{2}}x^{\Delta}(t)=c.
\end{equation}
Solving equation \eqref{1 row iloraz} and using the boundary conditions $x(0)=0$ and $x(1)=1$,
\begin{equation}
\label{rownanie}
x(t)=
\frac{1}{2Q}\int\limits_{0}^{t}\tau\Delta\tau
-t\left(\frac{1}{2Q}\int\limits_{0}^{1}\tau\Delta\tau -1\right),
\end{equation}
where $Q:=\frac{\mathcal{F}_{1}}{\mathcal{F}_{2}}$.
Therefore, the solution depends on the time scale.
Let us consider two examples: $\mathbb{T}=\mathbb{R}$
and $\mathbb{T}=\left\{0,\frac{1}{2},1\right\}$.
With $\mathbb{T}=\mathbb{R}$, from \eqref{rownanie} we obtain
\begin{equation}
\label{rownanie w R}
x(t)=\frac{1}{4Q}t^{2}+\frac{4Q-1}{4Q}t,
\quad x^{\Delta}(t) = x^{\nabla}(t) = x'(t)=\frac{1}{2Q}t+\frac{4Q-1}{4Q}.
\end{equation}
Substituting \eqref{rownanie w R} into
$\mathcal{F}_{1}$ and $\mathcal{F}_{2}$ gives
$\mathcal{F}_{1}=\frac{12Q+1}{24Q}$ and
$\mathcal{F}_{2}=\frac{48Q^{2}+1}{48Q^{2}}$, that is,
\begin{equation}
\label{rownanie Q}
Q=\frac{2Q(12Q+1)}{48Q^{2}+1}.
\end{equation}
Solving equation \eqref{rownanie Q} we get
$Q\in\left\{\frac{3-2\sqrt{3}}{12},\frac{3+2\sqrt{3}}{12}\right\}$.
Because \eqref{iloraz} is a minimizing problem,
we select $Q=\frac{3-2\sqrt{3}}{12}$ and we get the extremal
\begin{equation}
\label{ex1:note:12}
x(t)=-(3+2\sqrt{3}) t^{2} + (4 + 2 \sqrt{3}) t.
\end{equation}
If $\mathbb{T}=\left\{0,\frac{1}{2},1\right\}$,
then from \eqref{rownanie} we obtain
$x(t)=\frac{1}{8Q}\sum\limits_{k=0}^{2t-1}k+\frac{8Q-1}{8Q}t$, that is,
\begin{equation*}
x(t)=
\begin{cases}
0, & \text{ if } t=0,\\
\frac{8Q-1}{16Q}, & \text{ if } t=\frac{1}{2},\\
1, & \text{ if } t=1.
\end{cases}
\end{equation*}
Direct calculations show that
\begin{equation}
\label{rozwiazanie dla hZ}
\begin{gathered}
x^{\Delta}(0)=\frac{x(\frac{1}{2})-x(0)}{\frac{1}{2}}=\frac{8Q-1}{8Q},
\quad x^{\Delta}\left(\frac{1}{2}\right)
=\frac{x(1)-x(\frac{1}{2})}{\frac{1}{2}}=\frac{8Q+1}{8Q},\\
x^{\nabla}\left(\frac{1}{2}\right)
=\frac{x(\frac{1}{2})-x(0)}{\frac{1}{2}}=\frac{8Q-1}{8Q},
\quad x^{\nabla}(1)=\frac{x(1)-x(\frac{1}{2})}{\frac{1}{2}}
=\frac{8Q+1}{8Q}.
\end{gathered}
\end{equation}
Substituting \eqref{rozwiazanie dla hZ} into the integrals
$\mathcal{F}_{1}$ and $\mathcal{F}_{2}$ gives
$$
\mathcal{F}_{1}=\frac{8Q+1}{32Q},
\quad \mathcal{F}_{2}=\frac{64Q^{2}+1}{64Q^{2}},
\quad Q=\frac{\mathcal{F}_{1}}{\mathcal{F}_{2}}=\frac{2Q(8Q+1)}{64Q^{2}+1}.
$$
Thus, we obtain the equation $64Q^{2}-16Q-1=0$.
The solutions to this equation are:
$Q\in \left\{\frac{1-\sqrt{2}}{8}, \frac{1+\sqrt{2}}{8}\right\}$.
We are interested in the minimum value $Q$, so we select
$Q = \frac{1-\sqrt{2}}{8}$ to get the extremal
\begin{equation}
\label{eq:dif:ex1}
x(t)
=\begin{cases}
0, & \hbox{ if } t=0,\\
1+\frac{\sqrt{2}}{2}, & \hbox{ if } t=\frac{1}{2},\\
1, &  \hbox{ if }t=1.
\end{cases}
\end{equation}
Note that the extremals \eqref{ex1:note:12} and \eqref{eq:dif:ex1} are different:
for \eqref{ex1:note:12} one has $x(1/2) = \frac{5}{4} + \frac{\sqrt{3}}{2}$.
\end{example}

We now present a problem where, in contrast with Example~\ref{ex:1:se},
the extremal does not depend on the time scale $\mathbb{T}$.

\begin{example}
\label{ex:2}
Consider the autonomous problem
\begin{equation}
\label{iloraz 2}
\begin{gathered}
\mathcal{L}(x)=\frac{\int\limits_{0}^{2}\left(x^{\Delta}(t)\right)^{2}\Delta t}{
\int\limits_{0}^{2}\left[x^{\nabla}(t)+\left(x^{\nabla}(t)\right)^{2}\right]\nabla t}
\longrightarrow \min,\\
x(0)=0, \quad x(2)=4.
\end{gathered}
\end{equation}
If $x$ is a local minimizer to \eqref{iloraz 2},
then the Euler--Lagrange equations must hold, i.e,
\begin{equation}
\label{1 row iloraz 2}
\frac{2}{\mathcal{F}_{2}}x^{\nabla}(t)
-\frac{\mathcal{F}_{1}}{\mathcal{F}_{2}^{2}}(2x^{\nabla}(t)+1) = c
\quad \text{ and } \quad
\frac{2}{\mathcal{F}_{2}}x^{\Delta}(t)
-\frac{\mathcal{F}_{1}}{\mathcal{F}_{2}^{2}}(2x^{\Delta}(t)+1) = c,
\end{equation}
where
$\mathcal{F}_{1}:=\mathcal{F}_{1}(x)
=\int\limits_{0}^{2}\left(x^{\Delta}(t)\right)^{2}\Delta t$
and $\mathcal{F}_{2}:=\mathcal{F}_{2}(x)
=\int\limits_{0}^{2}\left[x^{\nabla}(t)
+\left(x^{\nabla}(t)\right)^{2}\right]\nabla t$.
Choosing one of the equations of \eqref{1 row iloraz 2},
for example the first one, we get
\begin{equation}
\label{2 row iloraz 2}
x^{\nabla}(t)=\left(c+\frac{\mathcal{F}_{1}}{\mathcal{F}_{2}^{2}}\right)
\frac{\mathcal{F}_{2}^{2}}{2\mathcal{F}_{2}-2\mathcal{F}_{1}}.
\end{equation}
Using \eqref{2 row iloraz 2} with boundary conditions $x(0)=0$ and $x(2)=4$,
we obtain, for any given time scale $\mathbb{T}$, the extremal $x(t)=2t$.
\end{example}

In the previous two examples, the variational functional is given
by the ratio of a delta and a nabla integral. We now discuss
a variational problem where the composition is expressed by
the product of three time-scale integrals.

\begin{example}
Consider the problem
\begin{equation}
\label{ex product}
\begin{gathered}
\mathcal{L}(x)=
\left(\int\limits_{0}^{1} t x^{\Delta}(t) \Delta t\right)
\left(\int\limits_{0}^{1} x^{\Delta}(t)\left(1+t\right)\Delta t\right)
\left(\int\limits_{0}^{1}\left(x^{\nabla}(t)\right)^{2}\nabla t\right)
\longrightarrow \min,\\
x(0)=0,\quad x(1)=1.
\end{gathered}
\end{equation}
If $x$ is a local minimizer to problem \eqref{ex product},
then the Euler--Lagrange equations must hold, and we can write that
\begin{equation}
\label{eq:el:ex3}
\left(\mathcal{F}_{1}\mathcal{F}_{3}+\mathcal{F}_{2}\mathcal{F}_{3}\right)t
+\mathcal{F}_{1}\mathcal{F}_{3}+2\mathcal{F}_{1}\mathcal{F}_{2}
x^{\nabla}(\sigma(t))=c,
\end{equation}
where $c$ is a constant,
$\mathcal{F}_{1}:=\mathcal{F}_{1}(x)
=\int\limits_{0}^{1} t x^{\Delta}(t) \Delta t$,
$\mathcal{F}_{2}:=\mathcal{F}_{2}(x)
=\int\limits_{0}^{1} x^{\Delta}(t)\left(1+t\right)\Delta t$,
and $\mathcal{F}_{3}:=\mathcal{F}_{3}(x)
=\int\limits_{0}^{1}\left(x^{\nabla}(t)\right)^{2}\nabla t$.
Using relation \eqref{polaczenie rozniczka delta}, we can write \eqref{eq:el:ex3} as
\begin{equation}
\label{ex produkt row}
\left(\mathcal{F}_{1}\mathcal{F}_{3}+\mathcal{F}_{2}\mathcal{F}_{3}\right)t
+\mathcal{F}_{1}\mathcal{F}_{3}+2\mathcal{F}_{1}\mathcal{F}_{2}x^{\Delta}(t)=c.
\end{equation}
Using the boundary conditions $x(0)=0$ and $x(1)=1$,
we get from \eqref{ex produkt row} that
\begin{equation}
\label{ex product equation}
x(t)=\left(1+Q\int\limits_{0}^{1}\tau\Delta \tau\right) t
-Q \int\limits_{0}^{t}\tau\Delta \tau,
\end{equation}
where $Q=\frac{\mathcal{F}_{1}\mathcal{F}_{3}
+\mathcal{F}_{2}\mathcal{F}_{3}}{2\mathcal{F}_{1}\mathcal{F}_{2}}$.
Therefore, the solution depends on the time scale.
Let us consider $\mathbb{T}=\mathbb{R}$ and $\mathbb{T}=\left\{0,\frac{1}{2},1\right\}$.
With $\mathbb{T}=\mathbb{R}$, expression \eqref{ex product equation} gives
\begin{equation}
\label{ex product w R}
x(t)=\left(\frac{2+Q}{2}\right) t - \frac{Q}{2}t^{2},
\quad x^{\Delta}(t)= x^{\nabla}(t) = x'(t) = \frac{2+Q}{2}-Qt.
\end{equation}
Substituting \eqref{ex product w R}
into $\mathcal{F}_{1}$, $\mathcal{F}_{2}$ and $\mathcal{F}_{3}$ gives:
$$
\mathcal{F}_{1}=\frac{6-Q}{12},
\quad \mathcal{F}_{2}=\frac{18-Q}{12},
\quad \mathcal{F}_{3}=\frac{Q^{2}+12}{12}.
$$
One can proceed by solving the equation $Q^{3} - 18 Q^{2} + 60 Q - 72=0$,
to find the extremal $x(t)=\left(\frac{2+Q}{2}\right) t - \frac{Q}{2}t^{2}$
with $Q = 2 \sqrt[3]{9+\sqrt{17}} + \frac{9-\sqrt{17}}{8} \sqrt[3]{(9+\sqrt{17})^2} + 6$.

Let us consider now the time scale $\mathbb{T}=\left\{0,\frac{1}{2},1\right\}$.
From \eqref{ex product equation} we obtain
\begin{equation}
\label{ex product w hZ}
x(t)=\left(\frac{4+Q}{4}\right)t-\frac{Q}{4}
\sum\limits_{k=0}^{2t-1}k
=
\begin{cases}
0, & \hbox{ if } t=0\\
\frac{4+Q}{8}, & \hbox{ if } t=\frac{1}{2}\\
1, &  \hbox{ if } t=1.
\end{cases}
\end{equation}
Substituting \eqref{ex product w hZ}
into $\mathcal{F}_{1}$, $\mathcal{F}_{2}$
and $\mathcal{F}_{3}$, we obtain
\begin{equation*}
\mathcal{F}_{1}=\frac{4-Q}{16},
\quad \mathcal{F}_{2}=\frac{20-Q}{16},
\quad \mathcal{F}_{3}=\frac{Q^{2}+16}{16}
\end{equation*}
and the equation $Q^{3}-18Q^{2}+48Q-96=0$. Solving this equation,
we find the extremal
\begin{equation*}
x(t)=
\begin{cases}
0, & \hbox{ if } t=0\\
\frac{5+\sqrt[3]{5}+\sqrt[3]{25}}{4}, & \hbox{ if } t=\frac{1}{2}\\
1, &  \hbox{ if } t=1.
\end{cases}
\end{equation*}
\end{example}

Finally, we apply the results of Section~\ref{sub:sec:iso:p}
to an isoperimetric variational problem.

\begin{example}
Let us consider the problem of extremizing
\begin{equation*}
\mathcal{L}(x)=
\frac{
\int\limits_{0}^{1}(x^{\Delta}(t))^{2}\Delta t}
{\int\limits_{0}^{1} tx^{\nabla}(t)\nabla t}
\end{equation*}
subject to the boundary conditions
$x(0)=0$ and $x(1)=1$, and the constraint
$$
\mathcal{K}(t)=\int\limits_{0}^{1} tx^{\nabla}(t)\nabla t=1.
$$
Applying Theorem~\ref{twierdzenie iso},
we get the nabla differential equation
\begin{equation}
\label{iso equation}
\frac{2}{\mathcal{F}_{2}}x^{\nabla}(t)
- \left(\lambda + \frac{\mathcal{F}_{1}}{(\mathcal{F}_{2})^{2}}\right) t = c.
\end{equation}
Solving this equation, we obtain
\begin{equation}
\label{iso rownanie}
x(t)=\left(1-Q\int\limits_{0}^{1}\tau\nabla\tau\right)t
+Q\int\limits_{0}^{t}\tau\nabla \tau,
\end{equation}
where $Q=\frac{\mathcal{F}_{2}}{2}\left(\frac{\mathcal{F}_{1}}{(\mathcal{F}_{2})^{2}}+\lambda\right)$.
Therefore, the solution of equation \eqref{iso equation} depends on the time scale.
As before, let us consider $\mathbb{T}=\mathbb{R}$ and $\mathbb{T}=\left\{0,\frac{1}{2},1\right\}$.

For $\mathbb{T}=\mathbb{R}$, we obtain from \eqref{iso rownanie} that
$x(t)=\frac{2-Q}{2}t+\frac{Q}{2}t^{2}$.
Substituting this expression for $x$ into the integrals
$\mathcal{F}_{1}$ and $\mathcal{F}_{2}$,
gives $\mathcal{F}_{1}=\frac{Q^{2}+12}{12}$ and $\mathcal{F}_{2}=\frac{Q+6}{12}$.
Using the given isoperimetric constraint,
we obtain $Q=6$, $\lambda =8$, and
$x(t)=3t^{2}-2t$.

Let us consider now the time scale $\mathbb{T}=\left\{0,\frac{1}{2},1\right\}$.
From \eqref{iso rownanie} we have
$$
x(t)=\frac{4-3Q}{4}t+Q\sum\limits_{k=1}^{2t}\frac{k}{4}
=
\begin{cases}
0, & \hbox{ if } t=0,\\
\frac{4-Q}{8}, & \hbox{ if } t=\frac{1}{2},\\
1, & \hbox{ if } t=1.
\end{cases}
$$
Simple calculations show that
$$
\mathcal{F}_{1}=\sum\limits_{k=0}^{1}\frac{1}{2} \left(x^{\Delta}\left(\frac{k}{2}\right)\right)^{2}
=\frac{1}{2}\left(x^{\Delta}(0)\right)^{2}
+\frac{1}{2}\left(x^{\Delta}\left(\frac{1}{2}\right)\right)^{2}=\frac{Q^{2}+16}{16},
$$
$$
\mathcal{F}_{2}=\sum\limits_{k=1}^{2}\frac{1}{4}k x^{\nabla}\left(\frac{k}{2}\right)
=\frac{1}{4} x^{\nabla}\left(\frac{1}{2}\right)+\frac{1}{2}x^{\nabla}(1)=\frac{Q+12}{16}
$$
and $\mathcal{K}(t)=\frac{Q+12}{16}=1$. Therefore, $Q=4$, $\lambda=6$, and we have the extremal
$$
x(t)=
\begin{cases}
0, &  \hbox{ if } t \in \left\{0, \frac{1}{2}\right\},\\
1, &  \hbox{ if } t=1.
\end{cases}
$$
\end{example}

% -------------------------------------

\section*{Acknowledgements}

This work was supported by {\it FEDER} funds through
{\it COMPETE} --- Operational Programme Factors of Competitiveness
(``Programa Operacional Factores de Competitividade'')
and by Portuguese funds through the
{\it Center for Research and Development
in Mathematics and Applications} (University of Aveiro)
and the Portuguese Foundation for Science and Technology
(``FCT --- Funda\c{c}\~{a}o para a Ci\^{e}ncia e a Tecnologia''),
within project PEst-C/MAT/UI4106/2011
with COMPETE number FCOMP-01-0124-FEDER-022690.
Dryl was also supported by FCT
through the Ph.D. fellowship SFRH/\linebreak
BD/51163/2010.
The authors would like to thank an anonymous referee
for careful reading of the submitted manuscript and
for suggesting several useful changes.

% -------------------------------------

% ------------------------------------------------

\end{document}